\documentclass{article}
\usepackage{amsfonts}
\usepackage{amssymb}
\usepackage{amsmath}
\usepackage{amsthm}
\usepackage{graphicx}

\usepackage{color} % load the color package (for hyperref)
   \definecolor{cites}{rgb}{0.50 , 0.00 , 0.00}  % colour for citations
   \definecolor{urls} {rgb}{0.00 , 0.00 , 0.50}  % colour for URL's
   \definecolor{links}{rgb}{0.00 , 0.00 , 0.50}   % colour for links

\usepackage{hyperref} % have cross-references in your pdf and dvi
\hypersetup{  % customize cross-references
      colorlinks=true,   % Yes, we want coloured links!
      citecolor=cites,   % Set the citations colour,
      urlcolor=urls,     % and the URL's colour,
      linkcolor=links,   % and the colour for the links.
      pdftitle={Circulant matrices: norm, powers, and positivity},  % Put your title here
      pdfauthor={Marko Lindner}, % And your name here
      pdfstartview=FitH,       % Fit, FitH (horizontal fit), FitV (vertical fit)
      bookmarksopen=false      % Bookmark tree already opened?
}

\newcommand\C{{\mathbb C}}
\newcommand\cC{{\mathcal C}}

\newcommand\F{{\bf F}}

\newcommand\N{{\mathbb N}}

\newcommand\T{{\mathbb T}}
\newcommand\R{{\mathbb R}}

\newcommand\Z{{\mathbb Z}}

\newcommand\ph{{\varphi}}
\renewcommand\th{{\vartheta}}
\newcommand\diag{\mathop{\rm diag}}

\renewcommand\Re{{\rm Re}}

\newcommand\bfx{{\bf x}}
\newcommand\bfy{{\bf y}}

\newcommand\bfO{{\bf 0}}
\newcommand\Ax{{\bf A}_\bfx}
\newcommand\Bx{{\bf B}_\bfx}
\newcommand\BxT{{\Bx^\top}}
\newcommand\Cx{{\bf C}_\bfx}
\newcommand\Cy{{\bf C}_\bfy}
\newcommand\CxT{{\Cx^\top}}
\newcommand\Cxadj{{\Cx^*}}
\newcommand\Dx{{\bf D}_\bfx}
\newcommand\Dxadj{{\Dx^*}}
\newcommand\bfD{{\bf D}}
\newcommand\bfM{{\bf M}}
\newcommand\bfN{{\bf N}}
\newcommand\bfR{{\bf R}}

\newtheorem{theorem}{Theorem}[section]
\newtheorem{lemma}[theorem]{Lemma}
\newtheorem{corollary}[theorem]{Corollary}

\newenvironment{example}
 {\par\noindent\refstepcounter{theorem}{\bf Example \thetheorem}}
 {\raisebox{1mm}{\framebox{}}\pagebreak[2]}

\numberwithin{figure}{section}  % needs package amsmath

\begin{document}
\title{\bf Circulant matrices: norm, powers, and positivity}
\author{
{\sc Marko Lindner}\footnote{Techn. Univ. Hamburg (TUHH), Institut Mathematik, D-21073 Hamburg, Germany, \url{lindner@tuhh.de}}
}

\date{\today}
\maketitle
\begin{quote}
~\\[-2em]
\footnotesize {\sc Abstract.}
In their recent paper ``The spectral norm of a Horadam circulant matrix'', %\cite{4Fins1}, 
Merikoski, Haukkanen, Mattila and Tossavainen
study under which conditions the spectral norm of a general real circulant matrix ${\bf C}$
equals the modulus of its row/column sum. We improve on their sufficient condition
until we have a necessary one. Our results connect the above problem to positivity of sufficiently high 
powers of the matrix ${\bf C^\top C}$. We then generalize the result to complex circulant matrices.
\end{quote}

\section{Introduction and Preliminaries}
For $n\in\N$ and $\bfx=(x_0,\dots,x_{n-1})\in\R^n$, look at the 
circulant matrix
\[
\Cx\ :=\ \begin{pmatrix}
     x_0&x_1&\cdots&x_{n-1}\\
     x_{n-1}&x_0&\ddots &\vdots\\[0mm]
 %    &x_{n-1}&x_0&\ddots&\\[-1mm]
     \vdots&\ddots&\ddots      &x_1\\
     x_1&\cdots&x_{n-1}&x_0
    \end{pmatrix}\ \in\ \R^{n\times n}.
\]
Motivated by studies of so-called Horadam or Fibonacci circulant matrices, 
the authors of \cite{4Fins1,4Fins2} ask in \cite{4Fins1} under which conditions 
the spectral norm of $\Cx$  equals $|x_0+x_1+\dots+x_{n-1}|$. % if $\bfx\in\R^n$.
We give a sufficient and a necessary condition.
Both have to do with the positivity of powers of $\CxT\Cx$.

If $\bfR:={\bf C}_{(0,1,0,\dots,0)}$ denotes the cyclic backward shift
$\bfR:(u_1,\dots,u_n)\mapsto(u_2,\dots,u_n,u_1)$, then\\[-2mm]
\[
\Cx=x_0\bfR^0+x_1\bfR^1+\dots+x_{n-1}\bfR^{n-1}=c(\bfR)
\quad\text{with}\quad
c(t) := x_0t^0+x_1t^1+\dots+x_{n-1}t^{n-1}.
\]
The polynomial $c$ is called the {\sl symbol} of $\Cx$.
Most of the time, we understand $c$ as a function on
\[
\T_n\ :=\ \{t\in\C:t^n=1\}=\{\omega^0,\omega^1,\dots,\omega^{n-1}\}
\qquad\text{with}\qquad
\omega:=\exp(\tfrac{2\pi}n{\sf i}).
\]
It is easy to see that $\bfR$ diagonalizes %unitarily
as $\bfR=\F\bfD\F^*$, where $\bfD=\diag(\omega^0,\dots,\omega^{n-1})$
and $\F$ is the so-called {\sl Fourier matrix} $\tfrac1{\sqrt n}\bigl(\omega^{jk}\bigr)_{j,k=0}^{n-1}$. 
Note that $\F$ is unitary, so that $\F^{-1}=\F^*$. 
Consequently,%\\[-2mm]
\[
\Cx=c(\bfR)=c(\F\bfD\F^*)=\F\, c(\bfD)\,\F^*=\F\diag(c(\omega^0),\dots,c(\omega^{n-1}))\F^*
=\F\Dx\F^*
\]
with $\Dx:=\diag(c(\omega^0),\dots,c(\omega^{n-1}))$.
Since $\F$ is an isometry of $\C^n$ with the Euclidean norm,\\[-2mm]
\begin{equation} \label{eq:diag}
\|\Cx\|\ =\ \|\F\Dx\F^*\|\ =\ \|\Dx\|\ =\ \max\bigl(|c(\omega^0)|,|c(\omega^1)|,\dots,|c(\omega^{n-1})|\bigr)\ =:\ \|c\|_\infty,
\end{equation}
where $\|\cdot\|$ denotes the spectral norm of a matrix; it is the matrix
norm that is induced by the Euclidean norm.
Of course, all of this is standard \cite{Davis:Circ}. The Fourier transform $\F$ turns the
convolution $\Cx$ into a multiplication $\Dx$. We are just fixing notations here.
%\pagebreak

The question of \cite{4Fins1} is essentially, under which conditions%\\[-0.5em]
\begin{equation} \label{eq:claim}
 \|\Cx\|=\|c\|_\infty%\stackrel?=
 \qquad\text{equals}\qquad
 |x_0+x_1+\dots+x_{n-1}|=|c(1)|=|c(\omega^0)|.
\end{equation}
So let
\[
\cC_n\ :=\ \bigl\{\bfx=(x_0,\dots,x_{n-1})\in\R^n\ :\ \|\Cx\|=|x_0+x_1+\dots+x_{n-1}|\bigr\}.
\]
%and deduce the following characterization from \eqref{eq:claim}.%\\[-2em]
%
%\begin{lemma} \label{lem:Cn}
%Let $n\in\N$ and $\bfx=(x_0,\dots,x_{n-1})\in\R^n$. Then the following are equivalent.\\[0.1em]
%%\begin{itemize}\itemsep0mm
%\begin{tabular}{ll}
%&(i)\quad $\bfx\in\cC_n$,\\[0.1em]
%&(ii)\quad $\|\Cx\|=|x_0+x_1+\dots+x_{n-1}|$,\\
%&(iii)\quad $|c(t)|$ assumes its maximum on $\T_n$ at $t=1=\omega^0$; in short, $\|c\|_\infty=|c(1)|$.
%%\end{itemize}
%\end{tabular}
%\end{lemma}
%
Looking at \eqref{eq:claim}, we see that 
\[
\bfx\in\cC_n \quad\iff\quad\|c\|_\infty=|c(1)|,\quad
\text{i.e.~$|c(\cdot)|$ assumes its maximum on $\T_n$ at $t=1=\omega^0$.}
\]
We will work with the latter condition in what follows.
We will also study the following subset of $\cC_n$ if $n\ge 2$. Let
\[
\cC_n'\ :=\ \bigl\{\bfx\in\cC_n\ :\ \max_{t\in\T_n\setminus\{1\}}|c(t)|<|c(1)|=\|c\|_\infty\bigr\}\ \subset\ \cC_n.
\]
While, for $\bfx\in\cC_n$, the maximum of $|c(\cdot)|$ in $\T_n$ is attained at $t=1$,
for $\bfx\in\cC_n'$ it is {\bf only} attained at $t=1$, so that $\Cx$ has a spectral gap 
between the two largest (in modulus) eigenvalues.
We start with a simple sufficient
condition for membership in $\cC_n$ and $\cC_n'$, respectively. Here we write
$\bfx\ge\bfO$ ($\bfx>\bfO$) or $\bfM\ge\bfO$ ($\bfM>\bfO$) if each entry of, 
respectively, the vector $\bfx$ or the matrix $\bfM$ is nonnegative (positive). 
\begin{lemma} \label{lem:x>=0} 
Let $n\ge 2$ and $\bfx\in\R^n$.

{\bf a) } If $\bfx\ge\bfO$ or $-\bfx\ge\bfO$ (i.e. $\pm\Cx\ge\bfO$)
then $\bfx\in\cC_n$. (This is \cite[Corollary 2]{4Fins1}.)

{\bf b) } If $\bfx>\bfO$ or $-\bfx>\bfO$ (i.e. $\pm\Cx>\bfO$)
then $\bfx\in\cC_n'$.
\end{lemma}
\begin{proof}
{\bf a) }
By triangle inequality, every $|c(t)|$ with $t\in\T_n$ is bounded as follows%\\[-0.5em]
\[
\left|c(t)\right|\ =\ \left|x_0+x_1t^1+\dots+x_{n-1}t^{n-1}\right|
 \ \le\ |x_0|+|x_1|+\dots+|x_{n-1}|
 \quad\text{since}\quad
 |t|=1.
\]
But this upper bound, and hence the maximum $\|c\|_\infty$, is attained by $|c(1)|=|x_0+\dots+x_{n-1}|$
as soon as all $x_k$ have the same sign, $\bfx\ge\bfO$ or $-\bfx\ge\bfO$. %phase in the complex plane.

{\bf b) } The statement can be derived by the Perron-Frobenius theorem but here is
a more elementary proof.
Let $\bfx>\bfO$. (The argument is similar for $-\bfx>\bfO$.)
 By {\bf a)}, we have $|c(1)|=\|c\|_\infty$. 
For every $t\in\T_n\setminus\{1\}$, it holds $|x_0+x_1t|<|x_0|+|x_1t|$
since $x_0,x_1>0$ and $1$ and $t$ have different directions in $\C$.
Consequently, noting that $|t|=1$,
\begin{align*}
|c(t)|\ &=\ \left|x_0+x_1t^1+\dots+x_{n-1}t^{n-1}\right|
\ \le\ \underbrace{|x_0+x_1t|}_{<\ |x_0|+|x_1t|}+|x_2t^2|+\dots+|x_{n-1}t^{n-1}|\\
& <\ |x_0|+|x_1|+|x_2|+\dots+|x_{n-1}|
\ =\ x_0+\dots+x_{n-1}
\ =\ c(1)\ =\ |c(1)|\ =\ \|c\|_\infty.\qedhere
\end{align*}
\end{proof}
This sufficient condition for membership in $\cC_n$ or $\cC_n'$ seems quite generous.
\cite{4Fins1} suggests the following improvement. Put
\begin{equation} \label{eq:Bx}
\Bx\ :=\ \CxT\Cx\ =\ \Cxadj\Cx\ =\ (\F\Dx\F^*\bigr)\kern-0.1em^*(\F\Dx\F^*)\ =\ \F\Dxadj\Dx\F^*
\ =\ \F\Ax\F^*
\end{equation}
with $\Ax:=\Dxadj\Dx=\diag(b(\omega^0),\dots,b(\omega^{n-1}))$, where 
\[
b(t) := \overline{c(t)}c(t)=|c(t)|^2
\quad\text{for all}\quad t\in\T_n,
\quad\text{so that}\quad
\|b\|_\infty:=\max_{t\in\T_n}|b(t)|=\max_{t\in\T_n}|c(t)|^2=\|c\|_\infty^2.
\]
Then $\Bx$ is again a real circulant matrix. Applying Lemma \ref{lem:x>=0} to $\Bx$ (in place of $\Cx$), we get:
\pagebreak
\begin{lemma} \label{lem:Bx>=0} 
Let $n\ge2$, $\bfx\in\R^n$ and put $\Bx:=\CxT\Cx$.

{\bf a) } If $\Bx\ge\bfO$ then $\bfx\in\cC_n$. (This is \cite[Theorem 4]{4Fins1}.)\\[-1.3em]

{\bf b) } If $\Bx>\bfO$ then $\bfx\in\cC_n'$. 
\end{lemma}
\begin{proof}
Recall that the symbol $b$ of $\Bx$ is related to the symbol $c$ of $\Cx$ by $b(t)=|c(t)|^2$
for all $t\in\T_n$. So $b$ assumes its maximum at the same point(s) as $|c(\cdot)|$ does.
For {\bf a)}, by Lemma \ref{lem:x>=0} a),
\[
\Bx\ge\bfO\quad\Rightarrow\quad
\|b\|_\infty=|b(1)|\quad\Rightarrow\quad
\|c\|_\infty^2=|c(1)|^2\quad\Rightarrow\quad
\|c\|_\infty=|c(1)|\quad\Rightarrow\quad
\bfx\in\cC_n.
\]
{\bf b) } By Lemma \ref{lem:x>=0} b), positivity $\Bx>\bfO$ implies that $|b(t)|<\|b\|_\infty$
for all $t\in\T_n\setminus\{1\}$. But then also $|c(t)|=|b(t)|^{1/2}<\|b\|_\infty^{1/2}=\|c\|_\infty$
for all $t\in\T_n\setminus\{1\}$. So $\bfx\in\cC_n'$.
\end{proof}

Note that the case $-\Bx\ge\bfO$ is impossible (unless $\bfx=\bfO$, in which case $\Bx=\bfO$) since the main diagonal
of $\Bx$ carries the entry $\|\bfx\|_2^2$.
\section{Iterating the argument until sufficient becomes necessary}
Looking at Lemmas \ref{lem:x>=0} and \ref{lem:Bx>=0}, the following questions seem natural:\\[0.5em]
%
%\begin{itemize}\itemsep-1mm
\begin{tabular}{l}
(Q1)~~Is the new condition $\CxT\Cx\ge\bfO$ substantially weaker than the old condition $\pm\Cx\ge\bfO$?\\[0.1em]
(Q2)~~Do we get a chain of increasingly weaker sufficient conditions if we repeat the argument?\\[0.1em]
(Q3)~~Does that chain end in a necessary condition?
\end{tabular}
%\end{itemize}

Let us address those questions, starting with (Q1):
It is easy to see that for $n\in\{1,2\}$, the two conditions are equivalent but for $n\ge3$ they differ.
Figure \ref{fig:table} below indicates that the quotient of their probabilities grows as $n$ grows.
As an example for $n=3$, look at $\bfx=(1,-2,-3)$, where
\[
 \Cx=\begin{pmatrix}
      1&-2&-3\\
      -3&1&-2\\
      -2&-3&1
     \end{pmatrix}\not\ge\bfO,
\qquad    
-\Cx\not\ge\bfO
\qquad\text{but}\qquad
\Bx:=\CxT\Cx=\begin{pmatrix}
                14&1&1\\
                1&14&1\\
                1&1&14
               \end{pmatrix}\ge\bfO.
\]
So Lemma \ref{lem:x>=0} is not strong enough to show $\bfx\in\cC_3$, 
i.e. $\|\Cx\|=|1-2-3|=4$, but 
Lemma \ref{lem:Bx>=0} is.
%\pagebreak

About (Q2): With $\Bx=\CxT\Cx$, let us now look at $\BxT\Bx$. 
But since $\BxT=\Bx$, one has $\BxT\Bx=\Bx^2$.
This is still a circulant, to which we can apply Lemma \ref{lem:x>=0}.
Then one can again multiply $\Bx^2$ with its transpose (itself) or just with $\Bx$ and continue like that.
\begin{theorem} \label{thm:sufficient}
Let $n\ge2$, $\bfx\in\R^n$ and $\Bx=\CxT\Cx$.

{\bf a) } If $\Bx^m\ge\bfO$ for some $m\in\N$ then $\bfx\in\cC_n$.\\[-1.3em]

{\bf b) } If $\Bx^m>\bfO$ for some $m\in\N$ then $\bfx\in\cC_n'$. 
\end{theorem}
\begin{proof}
For every $m\in\N$, we have, by \eqref{eq:Bx},
\begin{equation} \label{eq:Bx^m}
\Bx^m = \F\,\Ax^m\,\F^* = \F\,\diag_{k=0}^{n-1}b(\omega^k)^m\ \F^*,
\quad\text{ so that}\quad
\|\Bx^m\| = \max_{k=0}^{n-1}|b(\omega^k)|^m = \|b\|_\infty^m = \|c\|_\infty^{2m}.
\end{equation}
So $\Bx^m$ is a circulant matrix with symbol $t\mapsto b(t)^m=|c(t)|^{2m}$. 
It assumes its maximum at the same point(s) of $\T_n$ as $|c(\cdot)|$ does.
Now argue as in the proof of Lemma \ref{lem:Bx>=0}.
\end{proof}

Looking at $m=2^0,\; 2^1,\; 2^2, \dots$ and noting that $\bfM,\bfN\ge\bfO$ implies $\bfM\cdot\bfN\ge\bfO$, 
we get that%\\[-0.3em]
\[
\begin{array}{ccccccccccccc}
\pm\Cx\ge 0
&\Rightarrow& \Bx\ge{\bf 0}
& \Rightarrow& \Bx^2\ge {\bfO}
& \Rightarrow& \Bx^4\ge {\bfO}
& \Rightarrow& \Bx^8\ge {\bfO}
& \Rightarrow& \cdots
& \Rightarrow& \bfx\in\cC_n,\\[0.2em]
\pm\Cx> 0
&\Rightarrow& \Bx>{\bfO}
& \Rightarrow& \Bx^2> {\bfO}
& \Rightarrow& \Bx^4>{\bfO}
& \Rightarrow& \Bx^8> {\bfO}
& \Rightarrow& \cdots
& \Rightarrow& \bfx\in\cC_n'.
\end{array}
\]
To illustrate that these are indeed chains of increasingly weaker conditions,
let us approximately compute\footnote{using a Monte Carlo simulation with one million 
equally distributed points in the unit ball} 
the portion of the unit ball in $\R^n$ that satisfies the corresponding condition:
\[
{\footnotesize%\scriptsize
\begin{array}{|c||r|r|r|r|r|r|r|c||r|}
\hline
\rule{0pt}{1.1em}n&\pm\bfx>\bfO&\Bx>\bfO&\Bx^2>\bfO&\Bx^4>\bfO&\Bx^8>\bfO&\Bx^{16}>\bfO&\Bx^{32}>\bfO&\cdots&\bfx\in\cC_n'\\
\hline
\rule{0pt}{1.1em}
%n=1& 100.0\% & 100.0\%  &100.0\%  & 100.0\%  & 100.0\%  &100.0\% & 100.0\%  &\cdots&100.0\%\\
n=2&  50.0\% &  50.0\%  & 50.0\%  &  50.0\%  &  50.0\%  & 50.0\% &  50.0\%  &\cdots& 50.0\%\\
n=3&  25.0\% &  42.3\%  & 42.3\%  &  42.3\%  &  42.3\%  & 42.3\% &  42.3\%  &\cdots& 42.3\%\\
n=4&  12.5\% &  25.0\%  & 27.3\%  &  28.9\%  &  29.8\%  & 30.3\% &  30.5\%  &\cdots& 30.8\%\\
n=5&   6.3\% &  23.2\%  & 25.4\%  &  27.1\%  &  28.1\%  & 28.6\% &  28.9\%  &\cdots& 29.2\%\\
n=6&   3.1\% &  16.7\%  & 20.0\%  &  21.9\%  &  22.8\%  & 23.1\% &  23.3\%  &\cdots& 23.5\%\\
n=7&   1.6\% &  14.7\%  & 18.1\%  &  20.4\%  &  21.7\%  & 22.4\% &  22.8\%  &\cdots& 23.2\%\\
n=8&   0.8\% &  10.4\%  & 14.3\%  &  16.8\%  &  18.1\%  & 18.8\% &  19.2\%  &\cdots& 19.5\%\\
n=9&   0.4\% &  10.3\%  & 14.4\%  &  17.0\%  &  18.3\%  & 18.9\% &  19.2\%  &\cdots& 19.5\%\\
n=10&  0.2\% &   7.5\%  & 11.6\%  &  14.3\%  &  15.7\%  & 16.3\% &  16.6\%  &\cdots& 16.9\%\\
\vdots& \multicolumn{1}{|c|}{\vdots}&\multicolumn{1}{|c|}{\vdots} &\multicolumn{1}{|c|}{\vdots} &\multicolumn{1}{|c|}{\vdots} &\multicolumn{1}{|c|}{\vdots} & \multicolumn{1}{|c|}{\vdots} &\multicolumn{1}{|c|}{\vdots} &         &\multicolumn{1}{|c|}{\vdots}\\
n=20& 2^{-19}&  1.9\%  &  5.2\%  &    7.9\%  &   9.4\%  &  10.1\%  & 10.4\%  &\cdots& 10.7\%\\
\hline
\end{array}
}
\]
~\\[-3.3em]
\begin{figure}[h]
\caption{An approximate computation of the portion of points $\bfx\in\R^n$ of the unit ball 
(note that all conditions are invariant under scaling of $\bfx$) that satisfy the 
corresponding condition in the header. Reading from left to right, every row seems to grow -- in the limit -- 
up to the portion of the ball that belongs to $\cC_n'$.
This is a positive sign with respect to our question (Q3).}
\label{fig:table}
\end{figure}
~\\[-2em]

Finally, we turn to our question (Q3) about necessary conditions for membership in $\cC_n$ or $\cC_n'$.
Nonnegativity / positivity of powers of $\Bx$ is not necessary for membership in $\cC_n$
(see Example \ref{ex:CnvsCn'} below). But, assuming a spectral gap, i.e.~membership in $\cC_n'$, 
we get convergence of the power method and hence positivity of large powers of $\Bx$ (due to the special
structure of the corresponding eigenvector).
\begin{theorem} \label{thm:necessary}
If $\bfx\in\cC_n'$ 
then there exists an $m_0\in\N$ such that $\Bx^m>\bfO$ for all $m\ge m_0$.
\end{theorem}
\begin{proof}
Let $\bfx\in\cC_n'$ and abbreviate $|c(\omega^k)|=:c_k$ for $k=0,\dots,n-1$.
Then $\|c\|_\infty=c_0>c_1,\dots,c_{n-1}\ge0$. From \eqref{eq:Bx^m} we conclude
\begin{align}
\nonumber \frac {\Bx^m}{\|\Bx^m\|}\ 
&=\ \frac 1{c_0^{2m}}\F \diag(c_0^{2m},c_1^{2m},\dots,c_{n-1}^{2m}) \F^*
\ =\ \F \diag\Bigl(1,\bigl(\frac{c_1}{c_0}\bigr)^{2m},\dots,\bigl(\frac{c_{n-1}}{c_0}\bigr)^{2m}\Bigr) \F^*\\
\label{eq:1111}
&\to\ \F \diag(1,0,\dots,0) \F^*
\ =\ \frac1n\begin{pmatrix}1&\cdots&1\\[-0.5em]\vdots&&\vdots\\[-0.2em]1&\cdots&1\end{pmatrix}
\ >\ \bfO\qquad\text{as}\qquad m\to\infty,
\end{align}
so that $\Bx^m >\bfO$ for all sufficiently large $m\in\N$.
\end{proof}
The argument in the proof of Theorem \ref{thm:necessary} does not work if
$|c(\cdot)|$ attains its maximum in another or in more than one point on $\T_n$. 
The following example shows that, indeed, $\cC_n'$ cannot be replaced by $\cC_n$ 
in Theorem~\ref{thm:necessary}.
\begin{example} \label{ex:CnvsCn'}
Take $n=5$ and
$\Cx := \F\,\diag(1,0,1,1,0)\,\F^*$.
The diagonal has its maximum in the first but also in the
3rd and 4th position, so that $\bfx\in\cC_5\setminus\cC_5'$.
The first row of $\Cx$ is
$\bfx=(\tfrac35, \alpha, \beta, \beta, \alpha)$
with
$\alpha = \tfrac15(1+2\cos(\tfrac{4\pi}5))<0$ and
$\beta = \tfrac15(1+2\cos(\frac{2\pi}5))>0$,
so that $\Cx\not\ge\bfO$ and $-\Cx\not\ge\bfO$.
But also $\Bx^m\not\ge\bfO$ since $\Cx=\CxT=\Cx^m=\Bx^m$ for all $m\in\N$.
\end{example}

So for membership in $\cC_n'$, we have the following equivalence.
\begin{corollary}
Let $n\ge2$ and $\bfx\in\R^n$. Then the following are equivalent.\\[0.1em]
\begin{tabular}{ll}
&(i)\quad $\bfx\in\cC_n'$,\\[0.1em]
&(ii)\quad $\exists m\in\N: \Bx^m>\bfO$,\\
&(iii)\quad $\exists m_0\in\N: \forall m\ge m_0: \Bx^m>\bfO$.
\end{tabular}
\end{corollary}
\begin{proof}
$(ii)\Rightarrow (i)$ is Theorem \ref{thm:sufficient} b),
$(i)\Rightarrow (iii)$ is Theorem \ref{thm:necessary} b),
and $(iii)\Rightarrow (ii)$ is obvious.
\end{proof}

\section{Complex entries}
The case $\bfx\in\C^n$ is only slightly different. When we refer to $\cC_n$ or $\cC_n'$
now, we mean the corresponding subsets of $\C^n$.
In a complex version of Lemma \ref{lem:x>=0} a) it would be enough to have all entries of $\bfx$ of the same
phase, i.e.~on the same ray $\{rz:r\ge 0\}$ with some $z\in\C$. But for Lemma~\ref{lem:Bx>=0}~a),
that ray would again have to be the nonnegative real axis, because the main diagonal entries of $\Bx:=\Cx^*\Cx$
are always there. The other entries of $\Bx$ or $\Bx^m$ need not even be real, let alone nonnegative or positive. 

However, the proof of Theorem \ref{thm:necessary} shows that the entries of $\Bx^m$ are in a
certain neighborhood of the positive half axis if $\bfx\in\cC_n'$ (also for the complex version) 
and $m$ is sufficiently large. On the other hand, by the continuity of each function value $c(t)$ with respect to $\bfx$,
one can generalize Lemma \ref{lem:x>=0} to an appropriate neighborhood of the positive half axis:
\begin{lemma} \label{lem:suff_complex}
If $n\ge 2$ and $\bfx=(x_0,\dots,x_{n-1})\in\C^n$ is such that
at least two adjacent entries of $\bfx$ are nonzero and all phases
are close to zero, precisely, each
\begin{equation} \label{eq:phase0}
\ph_k:=\arg x_k\in (-\pi,\pi]
\qquad\text{is subject to}\qquad
|\ph_k|<\frac{\pi}{2n},
%\qquad\text{for}\qquad
%k=0,\dots,n-1,
\end{equation}
then $\bfx\in\cC_n'$.
\end{lemma}
\begin{proof}
We start with $n$ general complex numbers $z_0,\dots,z_{n-1}\in\C$ and put $\psi_k:=\arg z_k$,
which we put to zero if $z_k=0$. Then the following ``generalized law of cosines'' is easily verified.
\begin{align}
\nonumber
|z_0+\dots+z_{n-1}|^2\ &=
\ (z_0+\dots+z_{n-1})\overline{(z_0+\dots+z_{n-1})}
\ =\ \sum_{j,k=0}^{n-1}z_j\overline{z_k}\\
\label{eq:coslaw}
&=\ \sum_{j=0}^{n-1}|z_j|^2+2\kern-0.8em\sum_{\scriptsize\begin{array}{c}j,k=0\\j<k\end{array}}^{n-1}\kern-0.8em\Re(z_j\overline{z_k})
\ =\ \sum_{j=0}^{n-1}|z_j|^2+2\kern-0.8em\sum_{\scriptsize\begin{array}{c}j,k=0\\j<k\end{array}}^{n-1}\kern-0.8em|z_j||z_k|\cos(\psi_j-\psi_k)
\end{align}
Putting $z_k:=x_k$ from above, we have $\psi_k=\ph_k$ and hence
\begin{equation} \label{eq:c1}
|c(1)|^2\ =\ |x_0+\dots+x_{n-1}|^2
\ \stackrel{\eqref{eq:coslaw}}=
\ \sum_{j=0}^{n-1}|x_j|^2+2\kern-0.8em\sum_{\scriptsize\begin{array}{c}j,k=0\\j<k\end{array}}^{n-1}\kern-0.8em|x_j||x_k|\cos(\ph_j-\ph_k).
\end{equation}
Now take $t=\omega^\ell\in\T_n\setminus\{1\}$ with some $\ell\in\{1,\dots,n-1\}$ and put
$z_k:=x_kt^k$ in \eqref{eq:coslaw}. Then $\psi_k=\arg(x_kt^k)=\arg x_k+k\arg t=\ph_k+k\ell\th$ with
$\th:=\arg\omega=\tfrac{2\pi}n$. Plugging this into \eqref{eq:coslaw}, we get
\begin{equation} \label{eq:ct}
|c(t)|^2\ =\ |x_0t^0+\dots+x_{n-1}t^{n-1}|^2
\ \stackrel{\eqref{eq:coslaw}}=
\ \sum_{j=0}^{n-1}|x_j|^2+2\kern-0.8em\sum_{\scriptsize\begin{array}{c}j,k=0\\j<k\end{array}}^{n-1}\kern-0.8em|x_j||x_k|\cos\bigl(\ph_j-\ph_k+(j-k)\ell\th\bigr).
\end{equation}
By our assumption \eqref{eq:phase0}, all differences $\ph_j-\ph_k$ are in the interval 
$(-\frac\pi n,\frac\pi n)=:I_n$. Since the length of $I_n$ is $\th=\frac{2\pi}n$, 
%\begin{equation} \label{eq:choice}
\[
\ph_j-\ph_k+(j-k)\ell\th\quad
\ \left\{
\begin{array}{ll}
=\ph_j-\ph_k,&\text{if } (j-k)\ell\in n\Z,\\
\not\in I_n,&\text{otherwise},
\end{array}\right\}
\quad \text{both modulo }2\pi.
\]
%\end{equation*}
Moreover, $\cos x<\cos y$ whenever $x\not\in I_n$ and $y\in I_n$ (modulo $2\pi$).
Consequently, all cosines in \eqref{eq:c1} are larger than or equal to the corresponding cosines in \eqref{eq:ct}.
So $|c(1)|\ge|c(t)|$.

For our two adjacent $j,k$ with $x_j$ and $x_k$ nonzero, we have $j-k=-1$ and hence
$(j-k)\ell\not\in n\Z$, so that the corresponding term in \eqref{eq:c1} is strictly larger than 
in \eqref{eq:ct}. Hence, $|c(1)|>|c(t)|$.
\end{proof}
So it is already enough for $\bfx\in\cC_n'$ that each entry of $\bfx$ is
in a certain cone around the positive real half axis. By the same arguments
as in the real case, one can look at a power of $\Bx:=\Cx^*\Cx$, which
is again a circulant matrix, and check whether the entries of its first (or any) 
row satisfy \eqref{eq:phase0}.
\begin{theorem} \label{thm:complex}
Let $n\ge 2$ and $\bfx\in\C^n$. Then the following are equivalent.\\[-2.2em]
\begin{itemize} \itemsep-1mm
\item[(i)] $\bfx\in\cC_n'$,
\item[(ii)] $\exists m\in\N:$ at least two adjacent entries of the first row of $\Bx^m$ are nonzero and satisfy \eqref{eq:phase0},
\item[(iii)] $\exists m\in\N:$ all entries of the first row of $\Bx^m$ are nonzero and satisfy \eqref{eq:phase0},
\item[(iv)] $\exists m_0\in\N: \forall m\ge m_0:$ all entries of the first row of $\Bx^m$ are nonzero and satisfy \eqref{eq:phase0}.
\end{itemize}
\end{theorem}
\begin{proof}
The implications $(iv)\Rightarrow(iii)\Rightarrow(ii)$ are obvious. It remains to check $(ii)\Rightarrow(i)\Rightarrow(iv)$.

$(ii)\Rightarrow(i)$: Let $m\in\N$ be as in $(ii)$ and denote the circulant matrix $\Bx^m$ by $\Cy$.
By Lemma \ref{lem:suff_complex}, $\bfy\in\cC_n'$, i.e.~the symbol $b$ of $\Bx^m$
has its maximum at $1$ and only there. Arguing as in the proofs of Lemma \ref{lem:Bx>=0} 
and Theorem \ref{thm:sufficient}, the same holds for the symbol $c$ of $\Cx$, so that $\bfx\in\cC_n'$.

$(i)\Rightarrow(iv)$: Let $\bfx\in\cC_n'$. Following the proof of Theorem \ref{thm:necessary}
up to \eqref{eq:1111}, we see that, for all entries of $\Bx^m$, let us denote them by  $b_{jk}^{(m)}$,
we have the following limits as $m\to\infty$,
\[
\frac{b_{jk}^{(m)}}{\|\Bx^m\|}\to\frac 1n,
\quad\text{so that}\quad
\frac{|b_{jk}^{(m)}|}{\|\Bx^m\|}\to\left|\frac 1n\right|=\frac 1n
\quad\text{and hence}\quad
\frac{b_{jk}^{(m)}}{|b_{jk}^{(m)}|}=\frac{b_{jk}^{(m)}}{\|\Bx^m\|}\frac{\|\Bx^m\|}{|b_{jk}^{(m)}|}\to\frac 1n\cdot n=1,
\]
showing that $\arg b_{jk}^{(m)}\to 0$. It follows that, for all sufficiently large $m$, all entries of $\Bx^m$
are nonzero and subject to \eqref{eq:phase0}. This clearly implies $(iv)$. 
\end{proof}

\section{Conclusion}
Theorems \ref{thm:sufficient} and \ref{thm:necessary} are clearly not meant to give
efficient ways of computing the spectral norm of a generic real circulant matrix -- 
one cannot beat formula \eqref{eq:diag} in terms of the computational cost.
Rather than that, our theorems connect two apparently different questions to each other:\\
(i) whether $\|\Cx\|$ equals $|x_0+\dots+x_{n-1}|$, and (ii) eventual positivity
of the semigroup $(\Bx^m)_{m=0}^\infty$. 

In the complex case, one has the same results but instead of being real and positive,
the matrix entries of $\Bx^m$ only have to belong to a certain cone \eqref{eq:phase0}
around the positive half axis.

\end{document}